\documentclass[leqno,11pt]{article}

\usepackage[utf8]{inputenc}
\usepackage[T1]{fontenc}
\usepackage{cmbright}
\usepackage[english]{babel}

\usepackage{fancyhdr}

\usepackage[intlimits]{amsmath}
\usepackage{amssymb, amsfonts}
\usepackage{amsthm}

\usepackage{MnSymbol}
\usepackage{mathbbol}
\usepackage{bbm}
\usepackage{mathrsfs}

\usepackage[all]{xy}

\usepackage{graphicx}
\usepackage{color}
\usepackage[inline]{enumitem}

\numberwithin{equation}{section}

\newtheorem{theoremcounter}{theoremcounter}[section]
\newtheorem{thmstarcounter}{thmstarcounter}

\newtheorem{corollary}[theoremcounter]{Corollary}
\newtheorem{lemma}[theoremcounter]{Lemma}

\newtheorem{proposition}[theoremcounter]{Proposition}
\newtheorem{theorem}[theoremcounter]{Theorem}
\newtheorem{thmstar}[thmstarcounter]{Theorem}

\theoremstyle{definition}

\newtheorem{definition}[theoremcounter]{Definition}

\newtheorem{example}[theoremcounter]{Example}

\newtheorem{problem}[theoremcounter]{Problem}
\newtheorem{remark}[theoremcounter]{Remark}

\newcommand{\cA}{\ensuremath{\mathcal{A}}}

\newcommand{\cG}{\ensuremath{\mathcal{G}}}

\newcommand{\cK}{\ensuremath{\mathcal{K}}}

\newcommand{\cN}{\ensuremath{\mathcal{N}}}

\newcommand{\cP}{\ensuremath{\mathcal{P}}}

\newcommand{\cU}{\ensuremath{\mathcal{U}}}

\newcommand{\rE}{\ensuremath{\mathrm{E}}}

\newcommand{\rM}{\ensuremath{\mathrm{M}}}

\newcommand{\rU}{\ensuremath{\mathrm{U}}}

\newcommand{\rmd}{\ensuremath{\mathrm{d}}}

\newcommand{\veps}{\ensuremath{\varepsilon}}

\newcommand{\vphi}{\ensuremath{\varphi}}

\newcommand{\ol}{\overline}

\newcommand{\eqstop}{\ensuremath{\, \text{.}}}
\newcommand{\eqcomma}{\ensuremath{\, \text{,}}}

\newcommand{\NN}{\ensuremath{\mathbb{N}}}

\newcommand{\RR}{\ensuremath{\mathbb{R}}}
\newcommand{\CC}{\ensuremath{\mathbb{C}}}

\newcommand{\id}{\ensuremath{\mathrm{id}}}
\newcommand{\ra}{\ensuremath{\rightarrow}}

\newcommand{\hra}{\ensuremath{\hookrightarrow}}

\newcommand{\ot}{\ensuremath{\otimes}}

\newcommand{\Cstar}{\ensuremath{\mathrm{C}^*}}

\newcommand{\bo}{\ensuremath{\mathcal{B}}}

\newcommand{\supp}{\ensuremath{\mathop{\mathrm{supp}}}}
\newcommand{\im}{\ensuremath{\mathop{\mathrm{im}}}}

\newcommand{\Cstarred}{\ensuremath{\Cstar_\mathrm{red}}}

\newcommand{\cont}{\ensuremath{\mathrm{C}}}
\newcommand{\contb}{\ensuremath{\mathrm{C}_\mathrm{b}}}
\newcommand{\conto}{\ensuremath{\mathrm{C}_0}}
\newcommand{\contc}{\ensuremath{\mathrm{C}_\mathrm{c}}}

\newcommand{\Ltwo}{\ensuremath{{\offinterlineskip \mathrm{L} \hskip -0.3ex ^2}}}

\newcommand{\linfty}{\ensuremath{{\offinterlineskip \ell \hskip 0ex ^\infty}}}

\newcommand{\Ad}{\ensuremath{\mathop{\mathrm{Ad}}}}

\newcommand{\grpaction}[1]{\ensuremath{\stackrel{#1}{\curvearrowright}}}

\newcommand{\fb}{\partial_F}
\newcommand{\bs}{\backslash}
\newcommand{\contblu}{\contb^{\mathrm{lu}}}

\newcommand{\papertitle}{Cocompact amenable closed subgroups: weakly inequivalent representations in the left-regular representation}
\newcommand{\authors}{Sven Raum}

\begin{document}

\thispagestyle{empty}

\begin{center}
  \begin{minipage}[c]{0.9\linewidth}
    \vskip -2em
    \textbf{\LARGE
      Cocompact amenable closed subgroups:  \\
      weakly inequivalent representations in the left-regular representation}
    \hspace{2em} \\[0.5em]
    \mbox{by \authors}$^1$

  \end{minipage}
\end{center}

\vspace{1em}

\renewcommand{\thefootnote}{}
\footnotetext{$^1$The research leading to these results has received funding from the People Programme (Marie Curie Actions) of the European Union's Seventh Framework Programme (FP7/2007-2013) under REA grant agreement n°[622322].}
\footnotetext{Last modified on \today}

\begin{center}
\begin{minipage}{0.8\linewidth}
\textbf{Abstract}.
  We show that if $H \leq G$ is a cocompact amenable closed subgroup of a unimodular locally compact group, then the reduced group $\Cstar$-algebra of $G$ is not simple.  Equivalently, there are unitary representations of $G$ that are weakly contained in the left-regular representation, but not weakly equivalent to it.  We discuss applications of this result and pose the problem to construct non-discrete topologically simple groups with a cocompact amenable closed subgroup but without a Gelfand pair.
\end{minipage}
\end{center}


\section{Introduction}
\label{sec:introduction}

If $G$ is a connected semisimple Lie group then the principal series representations of $G$ give rise to weakly inequivalent representations.  This amounts to saying that the reduced group $\Cstar$-algebra of $G$ is not simple.  Using structure theory of locally compact groups, we deduced in \cite{raum15-powers}[Theorem A] that every locally compact group that is not totally disconnected satisfies the same conclusion.

If $G = KP$ is the Iwasawa decomposition of a connected semisimple Lie group, the principal series of $G$ can be constructed from the action on its Furstenberg boundary $G \grpaction{} G/P$.  As a matter of fact, the parabolic subgroup $P \leq G$ is a maximal cocompact amenable closed subgroup of $G$.  So recent progress in the study of totally disconnected groups makes the natural question arise in how far the presence of a cocompact amenable closed subgroup of a locally compact group implies non-simplicity of its reduced group \Cstar-algebra.  In this article we confirm for the unimodular case the intuition that this should be the case.
\begin{thmstar}
  \label{thm:main}
  Let $G$ be a unimodular locally compact second countable group containing a cocompact amenable closed subgroup.  Then $\Cstarred(G)$ is not simple.
\end{thmstar}

In order to discuss examples to which our main theorem applies, let us restrict our attention to totally disconnected groups whose amenable radical, that is the maximal normal amenable closed subgroup, is trivial.  It is easy to see that a non-trivial amenable radical implies non-\Cstar-simplicity.  In the literature on totally disconnected groups, there appear two classes of groups admitting a cocompact amenable closed subgroup.  In \cite{capracemonod15-bieberbach} groups acting properly and cocompactly on a CAT(0)-space are studied (CAT(0)-groups).  The work of Caprace and Monod implies that a totally disconnected CAT(0)-group with trivial amenable radical and with a cocompact amenable closed subgroup is a product of the following groups.
\begin{itemize}
\item Closed subgroups of the automorphisms of a tree acting 2-transitively on the boundary.
\item Semisimple algebraic groups over non-Archemedian local fields.
\end{itemize}
Further, in \cite{capracedecornuliermonodtessera12} non-discrete hyperbolic groups were studied.  \cite[Theorem 8.1]{capracedecornuliermonodtessera12} says that every totally disconnected hyperbolic group with trivial amenable radical that contains a cocompact amenable closed subgroup is isomorphic to a closed subgroup of the automorphisms of a tree that acts 2-transitively on the boundary.  To the best of our knowledge, these two results cover all known examples of totally disconnected groups with trivial amenable radical and with a cocompact amenable closed subgroup.

It is not difficult to see that all examples previously mentioned admit a Gelfand pair, that is there is a compact open subgroup $K \leq G$ such that the Hecke algebra of $K$-bi-invariant continuous and compactly supported functions is commutative.  This implies on the nose that such a group is not \Cstar-simple.  However, there is no direct proof of the existence of a Gelfand pair in CAT(0)-groups and in hyperbolic groups admitting a cocompact amenable closed subgroup.  So the passage via Caprace-Monod's and Caprace-de Cornulier-Monod-Tessera's classification theorem is necessary to derive non-\Cstar-simplicity in these cases.  The operator algebraic approach present in this article comes to this conclusion directly.  The previous discussion motivates the following two problems.

\begin{problem}
  Find examples of non-discrete topologically simple groups with a cocompact amenable closed subgroup but without a Gelfand pair!
\end{problem}

\begin{problem}
  Give a direct proof that a CAT(0)-group with trivial amenable radical and with a cocompact amenable closed subgroup admits a Gelfand pair!
\end{problem}

At the end of this introduction let us describe the strategy used to prove Theorem \ref{thm:main}.  Since the work of Kalantar-Kennedy \cite{kalantarkennedy14-boudaries} and Breuillard-Kalantar-Kennedy-Ozawa \cite{breuillardkalantarkennedyozawa14} it is known that the Furstenberg boundary $\fb G$ of a group $G$ plays a crucial role for \Cstar-simplicity of discrete groups.  In \cite{kalantarkennedy14-boudaries} it is proved that a discrete group $G$ is \Cstar-simple if and only if $\cont(\fb G) \rtimes_r G$ is simple.  If $G$ is a locally compact group admitting a cocompact amenable closed subgroup, then its Furstenberg boundary is homogeneous and Green's imprimitivity theorem implies that the crossed product $\cont(\fb G) \rtimes G$ is not simple (Corollary \ref{cor:crossed-product-non-simple}).  Our main technical result shows that, for unimodular locally compact groups, simplicity of $\Cstarred(G)$ implies simplicity of the crossed product $\cont(\fb G) \rtimes_r G$.  Theorem \ref{thm:main} is right away derived from this fact.
\begin{thmstar}[Theorem \ref{thm:relating-simplicity}]
  Let $G$ be a unimodular locally compact second countable group.  If $\Cstarred(G)$ is simple, then $\cont(\fb G) \rtimes_r G$ is simple.
\end{thmstar}
This theorem's proof is inspired by methods employed in \cite{kalantarkennedy14-boudaries} and \cite{ozawa14-boundaries-lecture-notes} to characterise discrete \Cstar-simple groups.  They are combined with strategies used in \cite{raum15-powers} developed to study group \Cstar-algebras of totally disconnected groups.  The main novelties in this article are a version of the Hahn-Banach extension theorem for specific weights and a rigidity result for compact spaces with an action of a discrete hypergroup.

\subsection*{Acknowledgements}

We want to thank Matthew Kennedy for helpful conversations on \Cstar-simplicity of locally compact groups and for useful remarks on a draft version of this article.  Further thanks is due to Emmanuel Breuillard for inspiring discussions and asking whether the statement of the main theorem could be true.  We finally we want to thank Pierre-Emmanuel Caprace for pointing out to us the classification of CAT(0)-groups containing a cocompact amenable closed subgroup.

\section{Preliminaries}
\label{sec:preliminaries}

\subsection{Totally disconnected groups}
\label{sec:td-groups}

Let $G$ be a topological group.  We denote the connected component of the identity in $G$ by $G^0$.  We call $G$ totally disconnected if $G^0 = \{e\}$.  In this case, $G$ admits a neighbourhood basis of the identity consisting of compact open subgroups.  For any topological group $G$, there is a short exact sequence $1 \ra G^0 \ra G \ra G/G^0 \ra 1$, showing that $G$ is totally disconnected-by-connected.

If $G$ is locally compact, then there is an up to scaling unique left-invariant Borel measure $\mu$ on $G$.  Such a measure is called Haar measure of $G$.  The modular function of $G$ is defined by the formula $\mu( A g) = \Delta(g) \mu(A)$ for all $A \subset G$ measurable and $g \in G$.  It is a well-defined homomorphism $\Delta: G \ra \RR_{> 0}$.

\subsection{Group C*-algebras}
\label{sec:group-cstar-algebras}

\textit{The reduced group \Cstar-algebra}. If $G$ is a locally compact group, then after choice of a Haar measure, $\contc(G)$ becomes a *-algebra equipped with the convolution product $x * y (h) = \int_G x(g) y(g^{-1} h ) \rmd g$.  The left-regular representation $G \ra \cU(\Ltwo(G))$ induces a *-representation of $\contc(G)$.  We call the norm closure of $\contc(G)$ inside $\bo(\Ltwo(G))$ the reduced group \Cstar-algebra of $G$.  The canonical unitaries in $\rM(\Cstarred(G))$ are denoted by $u_g := \lambda(g)$.

\textit{Averaging projections}.  If $G$ is a totally disconnected group with Haar measure $\mu$, the averaging projections $p_K := \frac{1}{\mu(K)} \int_K u_k \rmd \mu(k)$ lie in $\Cstarred(G)$.  They form an approximate unit consisting of projections.  In particular, the *-algebra $\bigcup_{K \leq G \text{ compact open}} p_K \Cstarred(G) p_K$ is dense in $\Cstarred(G)$.

\textit{Modular automorphism group}.  If $\Delta$ is the modular function of $G$, then $(\rU_t\xi)(g) := \Delta(g)^{it} \xi(g)$ defines a unitary $U_t$ on $\Ltwo(G)$.   The one-parameter automorphism group $(\Ad U_t)_t$ restricts to $\Cstarred(G)$, where it is denoted by $(\sigma_t)_t$.

\textit{Weights}.  A weight on a \Cstar-algebra $A$ is a function $\vphi:A^+ \ra \RR_{\geq 0} \cup \{\infty\}$ satisfying
$\vphi(\lambda(x + y)) = \lambda(\vphi(x) + \vphi(y))$ for all $\lambda > 0$ and $x, y \in A^+$.  Every choice of a Haar measure $\mu$ on a locally compact group $G$ defines a so called Plancherel weight on $\Cstarred(G)$.  It is densely defined and lower semi-continuous for the norm topology.  If $G$ is unimodular, every Plancherel weight $\vphi$ is tracial in the sense that $\vphi(x^*x) = \vphi(xx^*)$ for all $x \in \Cstarred(G)$.   In case $G$ is totally disconnected, Plancherel weights can be characterised as lower-semicontinuous KMS-weights $\vphi$ with respect to $(\sigma_t)_t$ satisfying the formula
\begin{equation*}
  \vphi(p_K u_g) = \mathbb 1_K(g) \frac{1}{\mu(K)}
\end{equation*}
for every compact open subgroup $K \leq G$ and every $g \in G$.  In particular, if $K \leq G$ is a compact open subgroup then $\vphi(p_K \cdot ) = \vphi( \cdot p_K) = \vphi(p_K \cdot p_K)$ defines a positive linear functional on the compression $p_K \Cstarred(G) p_K$.  We obtain a well-defined positive functional on $\bigcup_K p_K \Cstarred(G) p_K$, which is tracial if $G$ is unimodular.

\textit{Conditional expectation}. If $G \grpaction{} X$ is a locally compact group acting by homeomorphisms on a compact space $X$ and $K \leq G$ is a compact open subgroup, we obtain a canonical conditional expectation $p_K( \cont(X) \rtimes_r G )p_K \ra \cont(K \bs X)$ which restricts to the normalised Plancherel weight on $p_K \Cstarred(G) p_K$ and to the canonical isomophism $\cont(K \bs X) p_K \cong \cont(K \bs X)$.  This conditional expectation is faithful.

\subsection{Unbounded completely positive maps}
\label{sec:unbounded-cp-maps}

We refer the reader to the work of Evans \cite{evans75} for details on unbounded completely positive maps.

\textit{Facial subalgebra}.  Given a \Cstar-algebra $A$ a subalgebra $\cA \subset A$ is called facial if $\cA$ is the span of $\cA \cap A^+$.

\textit{Unbounded completely positive map}. Fixing a facial subalgebra $\cA \subset A$, an unbounded completely positive map into a \Cstar-algebra $B$ and with domain $\cA$ is a linear map $\Phi: \cA \ra B$ such that $\Phi_n: \cA \ot \rM_n(\CC) \ra B \ot \rM_n(\CC)$ is positive for all $n \in \NN^\times$.

\textit{Restrictions of unbounded cp maps}.  Assume that there a densely defined unbounded completely positive map $\vphi: \cA \ra B$ on $A$ and some element in $a \in \cA^+$.  Then $\vphi|_{a \cA a}$ extends uniquely to a well-defined completely positive map on $a A a$.

\section{Non-simplicity of the crossed product}
\label{sec:non-simplicity-crossed-product}

\subsection{The Furstenberg boundary}
\label{sec:furstenberg-boundary}

Let $G$ be a locally compact group.  A compact $G$-space $X$ is called
\begin{itemize}
\item \emph{minimal} if $Gx \subset X$ is dense for all $x \in X$;
\item \emph{strongly proximal} if for every $\mu \in \cP(X)$ the set $\ol{G\mu}$ contains a point-mass.
\end{itemize}
A minimal and strongly proximal compact $G$-space is called a $G$-boundary.

Furstenberg proved that there exists a maximal $G$-boundary for every locally compact group $G$.  This is the \emph{Furstenberg boundary} of $G$, denoted by $\fb G$ \cite{furstenberg72}.  The Furstenberg boundary is rigid in the sense that every $G$-equivariant unital completely positive map ($G$-ucp map for short) $\cont(\fb G) \ra \cont(\fb G)$ is the identity map.

In \cite{kalantarkennedy14-boudaries}, Kalantar-Kennedy showed that the Furstenberg boundary is $G$-injective in the sense that $\cont(\fb G)$ is an injective object in the category of $G$-operator systems.  Further, every choice of a probability measure on $\fb G$ defines a Poisson maps embedding $\cont(\fb G) \ra \contblu(G)$ in a $G$-equivariant way.  We obtain a $G$-equivariant conditional expectation $\contblu(G) \ra \cont(\fb G)$ for each such embedding.  For the  rest of the article, we fix such an embedding $\cont(\fb G) \ra \contblu(G)$.

Ozawa proved the following theorem, identifying the Furstenberg boundary of $G$ with a homogeneous space in caes $G$ admits a cocompact amenable closed subgroup.
\begin{theorem}[Ozawa {\cite[Proposition 10]{ozawa14-boundaries-lecture-notes}}]
  \label{thm:homogeneous-furstenberg-boundary}
  Let $G$ be a locally compact group containing a closed cocompact relatively amenable subgroup.  Then $\partial_F G \cong G/H$ for every maximal closed relatively amenable subgroup $H \leq G$.
\end{theorem}

\subsection{Imprimitivity}
\label{sec:imprimitivity}

Green's imprimitivity theorem describes a stable isomorphism between the crossed product with a homogeneous space and the group \Cstar-algebra of a point stabiliser.  Thanks to Theorem \ref{thm:homogeneous-furstenberg-boundary} it applies to the crossed product by the Furstenberg boundary in case there is a cocompact amenable closed subgroup.
\begin{theorem}[Green {[Corollary 2.10]\cite{green78}}]
  \label{thm:imprimivitivity}
  Let $G$ be a locally compact group and $H \leq G$ be a closed subgroup.  Then $\conto(G/H) \rtimes_r G \cong \Cstarred(H) \ot \cK(\Ltwo(G/H))$.
\end{theorem}

\begin{corollary}
  \label{cor:crossed-product-non-simple}
  Let $G$ be a non-trivial locally compact group containing a closed cocompact amenable subgroup.  Then $\cont(\partial_F G) \rtimes_r G$ is not simple.
\end{corollary}
\begin{proof}
  Since $G$ contains a cocompact amenable subgroup, it is exact by Skandalis' observation [Theorem in Section 7]\cite{kirchbergwassermann99}.  By Theorem \ref{thm:homogeneous-furstenberg-boundary} the Furstenberg boundary of $G$ is a homogeneous space $\fb G \cong G / H$ for some maximal relatively amenable closed subgroup $H \leq G$.  Caprace and Monod's \cite[Theorem 2]{capracemonod14} says that every closed relatively amenable subgroup of an exact group is amenable.  So $H$ is amenable.  In particular, $\Cstar(H)$ is not simple, since it contains the augmentation ideal induced from the trivial representation of $H$.  We can now apply Theorem \ref{thm:imprimivitivity} to conclude that 
  \begin{equation*}
    \cont(\fb G) \rtimes G \cong \cont(G/H) \rtimes G \cong \Cstar(H) \ot \cK(\Ltwo(G/H))
  \end{equation*}
  is not simple.  
\end{proof}

\section{The Furstenberg boundary of a quotient hypergroup}
\label{sec:furstenberg-boundary-quotient-hypergroup}

In this section we are going to study aspects of the Fustenberg boundary in the context of discrete hypergroups arising as double cosets of a an inclusion $K \leq G$ of a compact open subgroup into a locally compact group.  In this setting, Theorem \ref{thm:hypergroup-furstenberg-boundary} is an analogue of the rigidity of the Furstenberg boundary.

\subsection{Discrete hypergroups acting on operator systems}
\label{sec:actions-and-operator-systems}

\begin{definition}
  A discrete hypergroup is a discrete set $\cG$ with a unit element $e \in \cG$, a generalised inversion $\cG \ni g \mapsto \ol{g} \in \cG$ and an associative convolution product $*$ on probability measures on $\cG$ such that
  \begin{itemize}
  \item $g * e = e * g = g$ for all $g \in \cG$,
  \item $\ol{\ol{g}} = g$ for all $g \in \cG$,
  \item $\ol{g * h} = \ol{h} * \ol{g}$ (where we consider the linear extension of $g \mapsto \ol{g}$ to $\cP(\cG)$),
  \item $\supp (g * h) \ni e$ if and only if $h = \ol{g}$,
  \item $\supp(\mu * \nu)$ is finite for all finitely supported probability measures $\mu, \nu$ on $\cG$.
  \end{itemize}
\end{definition}
To simplify notation, we identify elements of a discrete hypergroup with the corresponding point masses.

\begin{example}[Quotient hypergroups]
 \label{ex:quotient-hypergroups}
 Our main example of discrete hypergroups arises from double cosets of an inclusion $K \leq G$ of a compact open subgroup into a locally compact group.  In this case $\cG = K \bs G / K$ becomes a discrete hypergroup when equipped with the structure
  \begin{itemize}
  \item $e_\cG = K$,
  \item $\ol{K g K} = K g^{-1} K$, and
  \item $(K g K) * (K h K) = \frac{1}{[K : K \cap  h K h^{-1}]} \sum_{l \in K / K \cap h K h^{-1}} \delta_{K g l hK}$. 
  \end{itemize}
\end{example}

\begin{definition}[Hypergroup operator systems]
Let $\cG$ be a discrete hypergroup.  A $\cG$-operator system is an operator system $S$ with a map $\alpha: \cG \ra \mathrm{UCP}(S)$ into the set of completely positive unital maps on $S$, such that
  \begin{itemize}
  \item $\alpha_e = \id_S$, and
  \item $\alpha_g \circ \alpha_h = \int_K \alpha_k \rmd (\delta_g * \delta_h)(k)$.
  \end{itemize}
We usually write $g x$ for $\alpha_g(x)$, if $g \in \cG$ and $x \in S$.

A map between two $\cG$ operator systems $\vphi: S \ra T$ is $\cG$-equivariant, if $\vphi(g x) = g \vphi(x)$ for all $g \in \cG$ and all $x \in S$.

If $X$ is a compact space, then a $\cG$-action on $X$ is per definition a $\cG$-operator system structure on $\cont(X)$.
\end{definition}

\begin{example}[Fixed point systems]
  Let $K \leq G$ be a compact open subgroup of a locally compact group and denote by $\cG = K \bs G /K$ the quotient hypergroup.  If $S$ is a $G$-operator system, then the fixed point space $S^K$, which is an operator subsystem, inherits the structure of a $\cG$-operator system by
  \begin{equation*}
    (KgK)x = \int_K k g x \, \rmd k
    \eqstop
  \end{equation*}
  If $\Phi: S \ra R$ is a $G$-ucp map between $G$-operator systems, then the restriction $\Phi: S^K \ra R^K$ is well-defined and $\cG$-equivariant.
\end{example}

\subsection{Injective operator systems over quotient hypergroups}
\label{sec:injective-operator-systems}

One of the main insights of \cite{kalantarkennedy14-boudaries} was that the Furstenberg boundary of a group $G$ defines an injective $G$-operator system.  A parallel result holds true for quotient hypergroups.  We follow Ozawa's proof from \cite[Theorem 6]{ozawa14-boundaries-lecture-notes} of Kalantar-Kennedy's result.

\begin{proposition}
  Let $G$ be a totally disconnected group and $K \leq G$ a compact open subgroup.  Denote by $\cG = K \bs G /K$ the quotient hypergroup.  Then $\cont(K \bs \fb G)$ is a $\cG$-injective operator system.
\end{proposition}
\begin{proof}
  We first prove that $\linfty(\cG)$ is $\cG$-injective.  If $S$ is a $\cG$-operator system, the evaluation $\mathrm{ev}_{K}: \linfty(\cG) \ra \CC$ induces a one-to-one correspondence between $\cG$-ucp maps  $S \ra \linfty(\cG)$ and states on $S$.   Given a state $\vphi: S \ra \CC$, we define the ucp map $\tilde \vphi: S \ra \linfty(\cG)$ by $\tilde \vphi(x) (g) : = \vphi(\ol{g} x)$.  We want to show that $\tilde \vphi(hx) = h (\tilde \vphi(x))$ for all $x \in S$ and all $h \in \cG$.
  \begin{equation*}
    (\tilde \vphi(hx))(g)
    =
    \vphi(\ol{g}hx)
    =
    \vphi(\ol{\ol{h}g} x)
    =
    (\tilde \vphi(x))(\ol{h}g)
    =
    (h (\tilde \vphi(x)))(g)
    \eqstop
  \end{equation*}
  So $\tilde \vphi$ is $\cG$-equivariant.  Given a $\cG$-ucp map $\Phi: S \ra \linfty(\cG)$, we obtain a state $\mathrm{ev}_e \circ \Phi$.  Further,
  \begin{equation*}
    \widetilde{(\mathrm{ev}_e \circ \Phi)}(x)(g)
    =
    (\mathrm{ev}_e \circ \Phi)(\ol{g} x)
    =
    \Phi(\ol{g} x)(e)
    =
    \Phi(x)(g)
    \eqcomma
  \end{equation*}
  by $\cG$-equivariance.  Now the Hahn-Banach theorem implies $\cG$-injectivity of $\linfty(\cG)$.

  Consider the composition $\cont(\fb G) \ra \contblu(G) \ra \linfty(G / K)$, which is a $G$-ucp map.  By $G$-rigidity of $\cont(\fb G)$ it follows injective.  So $\contblu(G) \ra \cont(\fb G) \hra \linfty(G/K)$ restricted to $\linfty(G/K)$ is a $G$-equivariant conditional expectation $\linfty(G/K) \ra \cont(\fb G) \subset \linfty(G/K)$.  Taking $K$-fixed points then yields a $\cG$-equivariant conditional expectation $\linfty(\cG) \ra \cont(K \bs \fb G)$.  This proves $\cG$-injectivity of $\linfty(\cG)$ and finishes the proof.
\end{proof}

\subsection{Boundary actions of discrete hypergroups}
\label{sec:boundary-actions-hyper-groups}

The following definition of minimal and strongly proximal actions of hypergroups will be justified by the rigidity results of Theorems~\ref{thm:hypergroup-rigidity}~and~\ref{thm:hypergroup-furstenberg-boundary}.
\begin{definition}
  Let $\cG$ be a hypergroup and $X$ a compact $\cG$-space.
  \begin{itemize}
  \item $\cG \grpaction{} X$ is minimal if $\ol{\mathrm{conv}}(\cG \delta_x) = \cP(X)$, for every $x \in X$.
  \item $\cG \grpaction{} X$ is strongly proximal if $\ol{\cG \mu} \cap X \neq \emptyset$ for every Borel probability measure $\mu \in \cP(X)$.
  \item A minimal and strongly proximal $\cG$-space is called a $\cG$-boundary.
  \end{itemize}
\end{definition}

\begin{proposition}
  Let $G$ be a totally disconnected group and $K \leq G$ a compact open subgroup.  Denote by $\cG = K \bs G /K$ the quotient hypergroup.  If $X$ is any $G$-boundary such that $K \bs X$ is totally disconnected, then $K \bs X$ is a $\cG$-boundary.
\end{proposition}
\begin{proof}
  Since $K \bs X$ is totally disconnected, every closed $K$-invariant subset of $X$ has a neighbourhood basis of $K$-invariant clopen sets.  So it suffices to show that if $\mu$ is an arbitrary $K$-invariant probability measure on $X$, then $\ol{G \mu}$ contains all point-masses.  But this is an immediate consequence of the fact that $X$ is a $G$-boundary.
\end{proof}

\begin{remark}[Orbit spaces of the Furstenberg boundary are totally disconnected]
  If $K \leq G$ is a compact open subgroup of a totally disconnected group, then $K \bs \fb G$ is extremally disconnected and hence totally disconnected.  Indeed, the $G$-equivariant expectation onto $\cont(\fb G) \subset \contblu(G)$ induces an expectation of the $K$-fixed points $\cont(K \bs \fb G) \subset \linfty(K \bs G)$.  So $\cont(K \bs \fb G)$ is an injective \Cstar-algebra and hence $K \bs \fb G$ extremally disconnected by \cite[Theorem 2.5]{gleason58}.
\end{remark}

\begin{theorem}
  \label{thm:hypergroup-rigidity}
  Let $\cG$ be a discrete hypergroup, $X$ a $\cG$-boundary and $\cG \grpaction{} Y$ a minimal $\cG$-space.  Then every $\cG$-ucp map $\cont(X) \ra \cont(Y)$ is injective.  If moreover, $\cont(X)$ is assumed to be $\cG$-injective, then $\cont(X)$ follows $\cG$-rigid.
\end{theorem}
\begin{proof}
  A $\cG$-ucp map $\vphi: \cont(X) \ra \cont(Y)$ corresponds to a $\cG$-equivariant map $\vphi^*:\cP(Y) \ra \cP(X)$.  Let $\mu \in \im \vphi^*$.  Then $\ol{\cG \mu} \cap X \neq \emptyset$ by strong proximality of $\cG \grpaction{} X$.  Further minimality of $X$ implies that $\vphi^*$ is surjective.  So $\vphi$ is injective.

Now assume that $\cont(X)$ is $\cG$-injective and let $\vphi: \cont(X) \ra \cont(X)$ be a $\cG$-ucp map.  Then $\vphi$ is injective by the first part of the proof.  Let $S = \vphi(\cont(X))$ be the image of $\vphi$, which is closed by the Stinespring dilation theorm.  Denote $\psi = \vphi^{-1}: S \ra \cont(X)$, which is a $\cG$-ucp map.  By $\cG$-injectivity of $\cont(X)$, there is a $\cG$-ucp extension of $\psi$:
\begin{equation*}
  \xymatrix{
    S \ar[d] \ar[r]^\psi & \cont(X) \\
    \cont(X) \ar[ur]^{\tilde \psi}
  }
  \eqstop
\end{equation*}
$\tilde \psi$ must be injective by the first part of the proof.  Since $\psi$ is surjective, this shows that $S = \cont(X)$.
\end{proof}

Summarising the results of Sections \ref{sec:injective-operator-systems} and \ref{sec:boundary-actions-hyper-groups}, we obtain the existence of a Furstenberg boundary for quotient hypergroups.
\begin{theorem}
  \label{thm:hypergroup-furstenberg-boundary}
  Denote by $\fb G$ the Furstenberg boundary of a totally disconnected group $G$.  Let $K \leq G$ be a compact open subgroup and $\cG = K \bs G /K$ the quotient hypergroup.  Then $\cont(K \bs \fb G)$ is a rigid and injective $\cG$-operator system.
\end{theorem}

\section{Relating $\Cstarred(G)$ and $\cont(\partial_F G) \rtimes_r G$}
\label{sec:relating-things}

This section is devoted to proving that for an arbitrary second countable unimodular locally compact group $G$, simplicity of $\Cstarred(G)$ implies simplicity of the crossed product with the Furstenberg boundary $\cont(\fb G) \rtimes_r G$ (Theorem \ref{thm:relating-simplicity}).  This is the only part in the proof of our main Theorem \ref{thm:main}, where we have to assume unimodularity.  It would be interesting to know whether this assumption is necessary in the statements of Theorem \ref{thm:relating-simplicity} and Theorem \ref{thm:main}.

Our main technical ingredient for the proof of Theorem \ref{thm:relating-simplicity} is the following version of the Hahn-Banach extension theorem for Plancherel weights.
\begin{lemma}
  \label{lem:hahn-banach}
  Let $G$ be a second countable totally disconnected locally compact group.  Let $A$ be a \Cstar-algebra containing $\Cstarred(G)$ in a non-degenerate way (that is $A \Cstarred(G) \subset A$ is dense).  Then every Plancherel weight on $\Cstarred(G)$ can be extended to a densely defined weight on $A$.
\end{lemma}
\begin{proof}
  Let $\vphi$ a Plancherel weight on $\Cstarred(G)$.  For a compact open subgroup $K \leq G$ we let $A_K := p_K A p_K$ and $\vphi_K$ the positive linear functional on $p_K \Cstarred(G) p_K$ obtained by restricting $\vphi$.  Since $G$ is second countable and totally disconnected, we can fix some countable neighbourhood basis $\cN$ of $\{e\}$ consisting of compact open subgroups.  We can assume that $\cN$ is strictly ordered by inclusion.  Let $U$ be the group generated by the symmetries $2p_K - 1$ for $K \in \cN$.  Since $\cN$ is ordered by inclusion, $U$ is an abelian group.  $[p_K, p_L] = 0$ for all compact open subgroups $K, L \leq G$, implies that $U$ acts by conjugation on $A_K$ and on $p_K \Cstarred(G) p_K$.  Since every compact open subgroup of $K \leq G$ is contained in the kernel of the modular function of $G$, we that $\vphi(p_K x) = \vphi(x p_K)$ for all $x \in \Cstarred(G)$.  Hence, the functionals $\vphi_K$ are $U$-invariant.  So the set of all extensions of $\vphi_K$ to a positive linear functional on $A_K$ is a $U$-invariant compact convex space.  So by the Markov-Kukatani theorem we find a $U$-invariant extension $\psi_{K,0}$ of $\vphi_K$ to $A_K$.  Since $\psi_{K,0}$ is $U$-invariant, it follows that $\psi_{K,0} (p_L x) = \psi_{K,0}(xp_L)$ for all $L \in \cN$ and all $x \in A_K$.

Since the state space of $A_K$ is compact, we can employ a diagonal sequence argument to pass to a subset of $\cN$ and assume that the sequence $(\psi_{L,0}|_{A_K})_{L \in \cN, L \leq K}$ converges for every $K$.  Denote the limit by $\psi_K$ and note that $\psi_K(p_L x) = \psi_K(x p_L)$ for every $L \in \cN$ and every $x \in A_K$.  Define $\psi = \limsup_{K \in \cN} \psi_K$, where each $\psi_K$ is viewed as a functional on $A$ after compressing with $p_K$.  Then $\psi: A^+ \ra \RR_{\geq 0} \cup \{\infty\}$ restricts to the Plancherel weight $\vphi$ on $\Cstarred(G)$.  Further $\psi(p_K x p_K) = \psi_K (p_K x p_K)$ for all $K \in \cN$ and all $x \in A^+$.  Since $p_K x p_K \ra x$ for all $x \in A^+$, it follows that $\psi$ is densely defined.

We show that $\psi$ is a weight.  If $x \in A^+$ and $\lambda \in \RR_{> 0}$, then 
\begin{equation*}
  \psi(\lambda x) = \limsup \psi_K( \lambda x ) = \lambda \limsup_K \psi_K(x) = \lambda \psi(x)
  \eqstop
\end{equation*}
Now let $x, y \in A^+$.  We have
\begin{equation*}
  \psi(x + y) = \limsup \psi_K(x + y) \leq \limsup \psi_K(x) + \limsup \psi_K(y) = \psi(x) + \psi(y)
  \eqstop
\end{equation*}
So it suffices to show that $\psi(x + y) \geq \psi(x) + \psi(y)$ in order to conclude that $\psi$ is a weight.  If $\psi(x) = \infty$, then
\begin{equation*}
  \psi(x + y) = \limsup \psi_K(x + y) \geq \limsup \psi_K(x) = \infty
  \eqcomma
\end{equation*}
by positivity of each $\psi_K$.  So $\psi(x + y) = \psi(x) + \psi(y)$ in the case $\psi(x) = \infty$ and by symmetry also in the case $\psi(y) = \infty$.  So we may assume that $\psi(x), \psi(y) < \infty$.  We show that the sequences $(\psi_K(x))_K$, $(\psi_K(y))_K$  and $(\psi_K(x + y))_K$ are convergent.
Let $K \in \cN$ and $\veps > 0$.  For small enough $L \in \cN$ contained in $K$, we have $\|p_K x^{1/2} p_L x^{1/2} p_K - p_K x p_K\| < \veps / \psi(p_K)$.  Using the facts that $p_L p_K = p_K$ and that $p_K$ is in the centraliser of $\psi_L$, we obtain the following estimate.
\begin{align*}
  \psi_K(x)
  & =
  \psi_K(p_K x p_K) \\
  & \leq
  \psi_K(p_K x^{1/2} p_L x^{1/2} p_K) + \veps \\
  & = 
  \psi_L(p_K x^{1/2} p_L x^{1/2} p_K) + \veps \\
  & = 
  \|p_L x^{1/2} p_K\|_{\psi_L, 2}^2 + \veps \\
  & \leq
  \|p_L x^{1/2} p_L\|_{\psi_L, 2}^2 \|p_K\|^2 + \veps \\
  & =
  \psi_L(p_L x^{1/2} p_L x^{1/2} p_L) + \veps \\
  & \leq
  \psi_L(p_L x p_L) + \veps\\
  & =
  \psi_L(x) + \veps
  \eqstop
\end{align*}
This is sufficient to show convergence of $(\psi_K(x))_K$.  The same argument implies convergence of $(\psi_K(y))_K$ and $(\psi_K(x + y))_K$.  This yields
\begin{equation*}
  \psi(x + y) = \lim_m \psi_m (x + y) = \lim_m \psi_m(x) + \lim_m \psi_m(y) = \psi(x) + \psi(y)
  \eqcomma
\end{equation*}
which finishes the proof of the lemma.
\end{proof}

\begin{theorem}
  \label{thm:relating-simplicity}
  Let $G$ be a unimodular locally compact  second countable group.  If $\Cstarred(G)$ is simple, then $\cont(\partial_F G) \rtimes_r G$ is simple.
\end{theorem}
\begin{proof}
  By \cite{raum15-powers}, simplicity of $\Cstarred(G)$ implies that $G$ is totally disconnected, so we restrict our attention to this setting.  Let $\pi: \cont(\partial_F G) \rtimes_r G \ra \bo(H)$ be a  *-representation, whose image we denote by $A$.  We have to show that $\pi$ is either zero or faithful.  Since $\Cstarred(G)$ is simple, either $\Cstarred(G) \subset \ker \pi$ or $\pi$ is injective on $\Cstarred(G)$.  The former case implies that $\pi$ is zero, since $(p_K)_K$, $K$ running over compact open subgroups of $G$, is an approximate unit for $\cont(\partial_F G) \rtimes_r G$ contained in $\Cstarred(G)$  (see \cite[Proposition 2.13]{raum15-powers}).  So we may assume that $\pi$ is injective on $\Cstarred(G)$ and identify it with its image in $A$.

By Lemma \ref{lem:hahn-banach}, there is a densely defined weight $\psi$ on $A$ such that $\psi|_{\Cstarred(G)}$ is a Plancherel weight.  $\cA = \bigcup_{K \leq G \text{compact open}} p_K A p_K$ is a facial subalgebra of $A$.  Unimodularity of $G$ implies that
\begin{equation*}
 \psi(u_g p_K x p_K u_g^*)
 =
 \psi(p_{gKg^{-1}} u_g x u_g^* p_{g K g^{-1}})
 \leq
 \psi(p_{gKg^{-1}}) \|u_g x u_g^*\|
 =
 \psi(p_K) \|x\|
 \eqcomma
\end{equation*}
for all $g \in G$, $x \in A^+$ and $K \leq G$ compact open.  So we obtain a densely define unbounded completely positive map on $A$ by 
\begin{equation*}
  \Psi_0: 
  \cA \ra \contb^{\mathrm{lu}}(G):
  \Psi_0(x)(g) = \psi(u_g^* x u_g)
  \eqstop
\end{equation*}
By definition $\Psi_0$ is $G$-equivariant and by unimodularity of $G$ we have $\Psi_0(p_K)(g) = \psi(p_{g^{-1}Kg}) = \psi(p_K)$ for all $g \in G$.

Denote by $\rE$ the $G$-equivariant conditional expectation from $\contblu(G)$ onto $\cont(\fb G)$.  Let $\Psi := \rE \circ \Psi_0 \circ \pi$.  Note that for every compact open subgroup $K \leq G$ we have $\Psi(p_K) = \psi(p_K) 1$, so $\Psi$ is non-zero on $\cA$.  If $K \leq G$ is a compact open subgroup and $x \in (\cont(\fb G) \rtimes_r G)^+$, then $\Psi(p_K x p_K) \leq \|x\| \Psi(p_K)$.  Scaling $\Psi$ by $\psi(p_K)$ it restricts to a ucp map $\Psi_K:p_K (\cont(\fb G) \rtimes_r G) p_K \ra \cont(K \bs \fb G)$.  Further the restriction of $\Psi_K$ to $\cont(K \bs \fb G) p_K$ is equivariant with respect to the natural $\cG = K \bs G /K$ action.  Theorem \ref{thm:hypergroup-furstenberg-boundary} then implies that $\Psi_K|_{\cont(K \bs \fb G)p_K}$ is the natural isomorphism $\cont(K \bs \fb G)p_K \cong \cont(K \bs \fb G)$.  It follows that $\cont(K \bs \fb G)$ is in the multiplicative domain of $\Psi_K$.  Since $\Psi_K$ also restricts to the canonical state of $p_K \Cstarred(G) p_K$.

Now consider the unscaled $\Psi$.  We show that for every compact open subgroup $K \leq G$, $\Psi$ restricts to a multiple of the natural conditional expectation $p_K (\cont(\fb G) \rtimes_r G) p_K \ra \cont(K \bs \fb G)$.  Since elements of the form $p_K f u_g p_K$ with $g \in G$ and $f \in \cont(\fb G)$ span a dense subset of $p_K (\cont(\fb G) \rtimes_r G) p_K$ it suffices to show
\begin{equation*}
  \Psi(p_K f u_g p_K) = \mathbb 1_K(g) \int_K \phantom{.}^k f \rmd k
  \eqstop
\end{equation*}
We may assume that $f$ is $L := K \cap g K g^{-1}$ invariant.  Then
\begin{align*}
  p_K f u_g p_K
  & =
  \frac{1}{[K : L]} \sum_{h \in K/L} u_h p_L f u_g p_K \\
  & =
  \frac{1}{[K : L]} \sum_{h \in K/L} \phantom{.}^h f u_{hg} p_K
  \eqstop
\end{align*}
Further we obtain
\begin{align*}
  \Psi(p_K f u_g p_K)
  & =
  \frac{1}{[K : L]} \sum_{h \in K/L}  \Psi(\phantom{.}^h f u_{hg} p_K) \\
  &  =
  \frac{1}{[K : L]} \sum_{h \in K/L}  \phantom{.}^h f \Psi( u_{hg} p_K) \\
  & =
  \frac{1}{[K : L]} \sum_{h \in K/L}  \phantom{.}^h f \mathbb 1_K (hg) \\
  & =
  \mathbb 1_K (g) \int_K \phantom{.}^k f \rmd k
\end{align*}
We conclude that $\Psi$ and hence also $\pi$ are faithful on $p_K \cont(\cont(\fb G) \rtimes_r G) p_K$.  Since $p_K$ is an approximate unit for $\cont(\fb G) \rtimes_r G$, we conclude that $\pi$ is faithful.
\end{proof}

\begin{remark}
  We expect that a converse to Theorem \ref{thm:relating-simplicity} holds.  However, the proof of this would necessarily involve new ideas.  Assuming that $\cont(\partial_F G) \rtimes_r G$ is simple for some unimodular locally compact group $G$, one is tempted to adapt the strategy of \cite{ozawa14-boundaries-lecture-notes} in order to prove simplicity of $\Cstarred(G)$.  However, non-unitality of $\Cstarred(G)$ is a major obstruction here.  Starting with a *-representation $\pi:\Cstarred(G) \ra \bo(H)$, it is not clear that the image $\pi(\Cstarred(G))$ admits any non-zero $G$-equivariant (unbounded) cp map into the Furstenberg boundary.  This highlights the importance of Lemma~\ref{lem:hahn-banach}, which is used to construct non-zero maps into the Furstenberg boundary when proving Theorem~\ref{thm:relating-simplicity}.
\end{remark}

\bibliographystyle{mybibtexstyle}
\bibliography{operatoralgebras}

\vspace{2em}
{\small \parbox[t]{200pt}
  {
    Sven Raum \\
    Westf{\"a}lische Wilhelmsuniversit{\"a}t M{\"u}nster \\
    Fakult{\"a}t Mathematik und Informatik \\
    Einsteinstra{\ss}e 62 \\
    D-48149 M{\"u}nster \\
    Germany \\
    {\footnotesize sven.raum@gmail.com}
  }
}

\end{document}